\documentclass{amsart} 
\usepackage{amssymb, latexsym, amsmath} 
\begin{document} 
	
	\theoremstyle{plain} 
	\newtheorem{theorem}{Theorem}
	\newtheorem{lemma}{Lemma} 
	\newtheorem{corollary}{Corollary} 
	\newtheorem{proposition}{Proposition} 
	
	\theoremstyle{definition} 
	\newtheorem{definition}{Definition}
	\newcommand{\beal}{\[ \begin{aligned}} 

	\newcommand{\eeal}{ \end{aligned} \]} 
	\newcommand{\een}{\end{eqnarray}} 
\newcommand{\End}{\mathrm{End}}
\newcommand{\ol}{\overline}
\newcommand{\Perm}{\mathrm{Perm}}

\title[SLB-YBE]{Skew left braces and the Yang-Baxter equation} 
\author{Lindsay N. Childs\\June 1, 2022}

\begin{abstract} We give a self-contained proof that a skew left brace yields a solution of the Yang-Baxter equation.
	\end{abstract}
\maketitle
\section{Introduction}
A  skew left brace is a set $B = (B, \circ, \cdot)$ with two group operations that satisfy the single compatibility condition
\[ (\#) \ \ \  x \circ(y \cdot z) = (x \circ y) \cdot x^{-1} \cdot (x \circ z)\]
The inverse of $x$ in $(B, \circ)$ is denoted $\overline{x}$ and in $(B, \cdot)$ by $x^{-1}$. One easily checks from (\#) that the two groups $(B, \circ)$ and $(B, \cdot)$ share a common identity element, $1$.

Skew left braces play a significant role in Hopf-Galois theory (see [CGK...21]), but they were first defined in [GV17] to yield set-theoretic solutions of the Yang-Baxter equation.  Such a solution is a function $R: B \times B \to B \times B$ on a set $B$ that satisfies the equation

\[(*): \ \ (R \times id)(id \times R)(R \times id)(a, b, c) = (id \times R)(R \times id)(id \times R)(a, b, c).\]
for all $a, b, c$ in $B$. 

This equation has been a question of considerable interest among algebraists since 1990 (motivated by [Dr92]).  Solutions of the YBE have been constructed in various settings during the past 25 years (e. g. [LYZ03], [Ru07], [CJO14]), [BCJ16]) but the only general descriptions of how a skew left brace yields a solution to the YBE appear in [GV17]  and in [Ba18]. 
  
In this note we present a straightforward proof that a skew left brace yields a solution $R: B \times B\to B \times B$ of the form $R(x, y) = (\sigma_x(y), \tau_y(x))$ where $\sigma_x(y)= x^{-1}(x \circ y)$ is the well-known $\lambda$-function  (or $\gamma$-function) associated to a skew brace, and $\tau_y(x)$ is defined by the equation that $\sigma_x(y) \circ \tau_y(x) = x \circ y$. The only facts needed for the proof are that $\sigma_x(\sigma_y(z)) = \sigma_{x \circ y}(z)$ and $\tau_y(\tau_x(z)) = \tau_{x \circ y}(z)$, both of which we prove.  The proof of the $\sigma$-result is from [GV17].  The $\tau$-result does not appear in [GV17]; there is a proof in [Bac18], but the proof here was obtained independently of [Bac18]. For the convenience of the reader, the proof that our solution of the YBE works is entirely self-contained.  

\section{The proof}
Given a skew brace $B  = (B, \circ, \cdot)$ , define $\sigma_x: B \to B $ by
\[\sigma_x(y) = x^{-1} \cdot (x \circ y)\]  

Define 
\[ \tau_y(x) = \overline{\sigma_x(y)} \circ x \circ y = \overline{ x^{-1} \cdot (x \circ y)} \circ x \circ y. \]
Then for all $x$, $y$ in $B$, 	$\sigma_x$ and $\tau_y$  are one-to-one maps from $B$ to $B$, and by definition of $\tau_y(x)$, $\sigma_x(y) \circ \tau_y(x) = \sigma_x(y) \circ (\overline{\sigma_x(y)} \circ x \circ y) = x \circ y.$	
Define 
\[ R: B \times B \to B \times B\] 
by
\[R(a, b) = (\sigma_a(b), \tau_b(a)) = (\sigma_a(b), \overline{\sigma_a(b)}\circ a \circ b ). \]
Then for all $a, b$ in $B$, if $R(a, b) = (s, t)$, then $s \circ t = \sigma_a(b) \circ \tau_b(a) =  a \circ b$.   

We will prove:

\begin{theorem}\label{p3} If $B$ is a skew left brace and $R:  B \times B \to B \times B$ is defined by $R(a, b) = (\sigma_a(b), \tau_b(a))$ for $a, b$ in $B$, then $R$ is a solution of the Yang-Baxter equation: for all $a, b, c$ in $B$, 
\[(*): \ \ (R \times id)(id \times R)(R \times id)(a, b, c) = (id \times R)(R \times id)(id \times R)(a, b, c).\]
\end{theorem}
Since $ \sigma_a$ and $\tau_b$ are one-to-one maps from $B$ to $B$ for all $a$, $b$ in $B$, the solution $R$ of the Yang-Baxter equation is nondegenerate.

\begin{proof}
 Given a skew brace $B(\circ, \cdot)$,  the maps $ \sigma_x(y) = x^{-1} \cdot (x \circ y)$ and $\tau_y(x) = \overline{\sigma_x(y)} \circ x \circ y$  satisfy the following two properties, as we show below: 

	(i): $ \sigma $ is a homomorphism from $(B, \circ )$ to $\Perm(B)$:    for $x$, $y$, $z$ in $B$, 
	\[ \sigma_{x \circ y}(z) = \sigma_x( \sigma_y(z));\]

	(ii): $\tau $ is an anti-homomorphism from $(B, \circ)$ to $\Perm(B)$: \[\tau_{z \circ y}(x) = \tau_y(\tau_z(x)).\]
		
Beside these two properties, the only other property we need is that if $R(u, v) = (\sigma_u(v), \tau_v(u)) = (y, z)$, then $u \circ v = y \circ z$, as noted above.
	
	These three properties suffice to show that $R$ satisfies
	\[(R \times 1)(1 \times R)(R \times 1)(a, b, c)  = (1 \times R)(R \times 1)(1 \times R)(a, b, c)  \ \ (*),\]  
	for all $a, b, c$ in $B$, as follows.	
	
	The left side of (*) is:
	\[ (R \times 1)(1 \times R)(R \times 1)(a, b, c) = (R \times 1)(1 \times R)(d, e, c) = (R \times 1)(d, f, g) = (h, k, g) \]
	where
	\[ d = \sigma_a(b), \ e = \tau_b(a), \text{ so } a \circ b = d \circ e,\]
	
	\[\ f = \sigma_e(c), \ g = \tau_c(e), \text{ so } \ e \circ c = f  \circ g,\]
and
	\[\ h = \sigma_d(f), \ k = \tau_f(d),  \text{ so } 
 \ d \circ f = h \circ k.\]

	The right side of (*) is:
	\[ (1 \times R)(R \times 1)(1 \times R)(a, b, c) = (1 \times R)(R \times 1)(a, q, r) = (1 \times R)(s, t, r)= (s, v, w),\] 
	where 
	\[ q = \sigma_b(c), \ r = \tau_c(b),  \text{ so }  b \circ c = q \circ r,\] 
	\[ \ s = \sigma_a(q), \  t = \tau_q(a),  \text{ so }   a \circ q = s \circ t,\]
and \[\ v = \sigma_t(r), \ w = \tau_r(t), \text{ so } \ t \circ r = v \circ w .\]

	We want to show that $(h, k, g) = (s, v, w)$.
	
	To show that $h = s$ uses property (i): $\sigma_{y \circ z}(x) = \sigma_y(\sigma_z(x))$, as follows:
\[ \begin{aligned} s &= \sigma_a(q) =\sigma_a(\sigma_b(c)) = \sigma_{a  \circ b}(c);\\
	h &= \sigma_d(f) = \sigma_d(\sigma_e(c)) = \sigma_{d \circ e}(c); \eeal
	and 
	\[d \circ e = \sigma_a(b) \circ \tau_b(a) = a \circ b.\]  
	So 
	\[h = \sigma_{d \circ e}(c) = \sigma_{a \circ b}(c) = s.\]
	To show that $w = g$ uses property (ii): $\tau_{z \circ y}(x) = \tau_y(\tau_z(x))$, as follows: 
	\beal g &= \tau_c(e) = \tau_c(\tau_b(a)) = \tau_{b \circ c}(a);\\
	w &= \tau_r(t) = \tau_r(\tau_q(a)) = \tau_{q \circ r}(a) \eeal
	and 
	\[ q \circ r = \sigma_b(c) \circ \tau_c(b) = b \circ c.\]  So 
	\[w = \tau_{q \circ r}(a) = \tau_{b \circ c}(a)  = g.\]
	
	Finally, to show that $k = v$ we just use many times that for any $u, v$,  if $R(u, v) =   (m, n)$, then $m \circ n =  u \circ v$, as follows:
	
	The left side of equation (*) is  $(h, k, g)$; the right side is $(s, v, w)$, and using all of the equalities above, we have that 
\[ s \circ v \circ w = a \circ b \circ c =h \circ k \circ g :\]
For
\[  s \circ (v \circ w) = s \circ (\sigma_t(r) \circ \tau_r(t)) = s \circ (t \circ r)\]
\[ = (s \circ t) \circ r = (\sigma_a(q) \circ \tau_q(a)) \circ r = (a \circ q) \circ r \]
\[ = a \circ (q \circ r) =  a \circ (\sigma_b(c) \circ \tau_c(b)) = a \circ (b \circ c);\]
while
\[ (a \circ b) \circ c =  (\sigma_a(b) \circ \tau_b(a)) \circ c = (d \circ e)  \circ c\]
\[= d \circ (e \circ c) = d \circ (\sigma_e(c) \circ \tau_c(e)) = d \circ (f \circ g) \]
\[ = (d \circ f) \circ g = (\sigma_d(f) \circ \tau_f(d)) \circ g = (h \circ k) \circ g.\]
So  $s \circ v \circ w = h \circ k \circ g$.
	Since $w = g$, and $h = s$ in the group $(B, \circ)$, it follows that $k = v$.  
	Given properties (i) and (ii), that completes the proof.
\end{proof}

To prove properties (i) and (ii) we need the following consequence of the compatibility condition (\#) for a skew brace (c.f. [GV17], Lemma 1.7 (2)):

\begin{lemma}\label{p0} For all $a$, $b$ in $B$, $ a^{-1} \cdot (a \circ b^{-1}) \cdot a^{-1} =(a \circ b)^{-1}$.\end{lemma}
\begin{proof}   The compatibility condition (\#) for a skew brace is that for all $x, y, z$ in $B$,
	\[ x \circ (y \cdot z) =( x \circ y) \cdot x^{-1} \cdot (x \circ z),\]
	hence
	\[ x \cdot (x \circ y)^{-1} \cdot (x \circ (y \cdot z)) = x \circ z\]
	or
	\[ x \circ z = x \cdot (x \circ y)^{-1} \cdot (x \circ (y \cdot z)).\]
	Set $x = a, y = b, z = b^{-1}$ to get  
	\[ a \circ b^{-1} = a \cdot (a \circ b)^{-1} \cdot a,\]
	or
	\[ a^{-1} \cdot (a \circ b^{-1}) \cdot  a^{-1} = (a \circ b)^{-1}.\]
\end{proof}

Here is property (i):  it is Proposition 1.9 (2) of [GV17].

\begin{proposition}\label{p1}  For all $x, y, z$ in $B$,   
	\[\sigma_{x \circ y}(z) = \sigma_x(\sigma_y(z)).\] \end{proposition}

\begin{proof} (from [GV17])
	The right side of 
	\[ \sigma_{x \circ y}(z) = \sigma_x(\sigma_y(z)) \] 
	is
	\beal  \sigma_x(\sigma_y(z)) &= x^{-1} \cdot (x \circ \sigma_y(z))\\
	&= x^{-1} \cdot (x \circ (y^{-1} \cdot (y \circ z)))\\
	&= x^{-1} \cdot (x \circ y^{-1}) \cdot  x^{-1} \cdot  (x \circ y \circ z) \ \ \ \text{(by (\#))}.\eeal
	By Lemma 1, this is
	\beal	&= (x \circ y)^{-1} \cdot (x \circ y \circ z)  \\
	&= \sigma_{x \circ y}(z).\eeal
\end{proof}

(We note that [GV17] proves that given a set $B$ with two group operations, $\cdot$ and $\circ$, and $\sigma_x(y) = x^{-1}\cdot(x \circ y)$, then for all $x, y, z$ in $B$,
\[\sigma_x(\sigma_y(z)) = \sigma_{x \circ y}(z)  \]
if and only if the compatibility condition $(\#)$ holds, if and only if  $B$ is a skew left brace:  see Proposition 1.9 of [GV17].)

Finally, we prove property (ii):
 
\begin{proposition}\label{p2}  $\tau $ is an anti-homomorphism from $(B, \circ)$ to $\Perm(B)$:
\[\tau_{y \circ z}(x) = \tau_z(\tau_y(x)).\] \end{proposition}

\begin{proof}
We begin with the definition of $\sigma_x(q)$:
\[ x^{-1}\cdot(x \circ y) = \sigma_x(y)\]
Rearrange the equation and use that $ x \circ y =\sigma_x(y) \circ \tau_y(x)$, to get: 
\[\sigma_x(y)^{-1}\cdot x^{-1} = (\sigma_x(y) \circ \tau_y(x))^{-1}\]
Apply the Lemma formula,  $(a \circ b)^{-1} = a^{-1}\cdot(a \circ b^{-1})\cdot a^{-1})$ 
to the right side, to get:
\[ \sigma_x(y)^{-1}\cdot x^{-1}= \sigma_x(y)^{-1}\cdot(\sigma_x(y) \circ \tau_y(x)^{-1})\cdot\sigma_x(y)^{-1}\]
Cancel $\sigma_x(y)^{-1}$ on the left and multiply both sides by $\cdot (x \circ y \circ z)$ on the right:
\[x^{-1}\cdot(x \circ y \circ z) = (\sigma_x(y) \circ \tau_y(x)^{-1})\cdot \sigma_x(y)^{-1}\cdot(x \circ y \circ z)\]     
Apply the definition of $\sigma$ to the left side and use that $x \circ y = \sigma_x(y) \circ \tau_y(x)$ on the right side:
\[ \sigma_x(y \circ z) =  (\sigma_x(y) \circ \tau_y(x)^{-1})\cdot \sigma_x(y)^{-1} \cdot ( \sigma_x(y) \circ \tau_y(x) \circ z)\] 
Apply the skew brace formula (\#) (in reverse):
\[ \sigma_x(y \circ z) =  (\sigma_x(y) \circ (\tau_y(x)^{-1}\cdot ( \tau_y(x) \circ z))\]  
Use the definition of $\sigma$ on the far right side:
\[	\sigma_x(y \circ z) = \sigma_x(y) \circ \sigma_{\tau_y(x)}(z) \]
Take the $\circ$-inverse of both sides, and multiply both sides by $ \circ x \circ  y \circ z$:
\[\ol{\sigma_x(y \circ z)} \circ x \circ y \circ z = \ol{\sigma_{\tau_y(x)}(z)} \circ (\ol{\sigma_x(y)} \circ x \circ y) \circ z\] 
Use the definition of $\tau$: $\tau_b(a) = \overline{\sigma_a(b)} \circ a  \circ b$ on the right side:
\[\ol{\sigma_x(y \circ z)} \circ x \circ (y \circ z) = \ol{\sigma_{\tau_y(x)}(z)} \circ \tau_y(x) \circ z,\] 
then on both sides:
\[\tau_{y \circ z}(x) = \tau_z(\tau_y(x))\]   
So $\tau$ is an anti-homomorphism on $(B, \circ)$. \end{proof}

\subsection*{References}
	
\
	
[BCJ16]  D. Bachiller, F. Cedo, E. Jespers, Solutions of the Yang-Baxter equation associated with a left brace, J. Algebra 463 (2016), 80-102. arXiv:1503.02814
	
\

[Ba18]	D. Bachiller, Solutions of the Yang-Baxter equation associated to skew left braces, with applications to racks, J. Knot Theory and its Ramifications 27 (2018). arXiv:1611.08138 

\

[CGK...21]  L. Childs, C. Greither, K. Keating, A. Koch, T. Kohl, P. Truman, R. Underwood, Hopf Algebras and Galois Module Theory, Mathematical Surveys and Monographs v. 260 (2021), American Mathematical Society.

\

[CJO14]  F. Cedo, E. Jespers, J. Okninski, Braces and the Yang-Baxter equation, Comm. Math. Physica 327 (2014), 101--116.  arXiv:1205.3587

\

[Ge92]  V. G. Drinfel'd, On some unsolved problems in quantum group theory, Lecture Notes in Math. vol 1510, Springer, Berlin, 1992, 1--6. 

\
	
[GV17]  L. Guarneri, L. Vendramin, Skew braces and the Yang-Baxter equation, Math. Comp 86 (2017), 2519--2534. arXiv:1511.03171

\

[LYZ03]  J. Lu, M. Yan, Y. Zhu, On the set-theoretical Yang-Baxter equation, Duke Math. J. 104 (2000), 153--170.

\

[Ru07]	W. Rump, Braces, radical rings, and the quantum Yang-Baxter equation, J. Algebra 307 (2007), 153--170.

\

[SV18]  A. Smoktunowicz, L. Vendramin, On skew braces, with an appendix by N. Byott and L. Vendramin, J. Comb. Algebra 2 (2018), 47--86.
\

\end{document}